\documentclass[reqno]{amsart}
\usepackage{amssymb}
\usepackage{a4wide}
\usepackage[colorlinks,citecolor=green,linkcolor=red]{hyperref}
\makeatletter
\def\@tocline#1#2#3#4#5#6#7{\relax
  \ifnum #1>\c@tocdepth 
  \else
    \par \addpenalty\@secpenalty\addvspace{#2}%
    \begingroup \hyphenpenalty\@M
    \@ifempty{#4}{%
      \@tempdima\csname r@tocindent\number#1\endcsname\relax
    }{%
      \@tempdima#4\relax
    }%
    \parindent\z@ \leftskip#3\relax \advance\leftskip\@tempdima\relax
    \rightskip\@pnumwidth plus4em \parfillskip-\@pnumwidth
    #5\leavevmode\hskip-\@tempdima
      \ifcase #1
       \or\or \hskip 1em \or \hskip 2em \else \hskip 3em \fi%
      #6\nobreak\relax
    \dotfill\hbox to\@pnumwidth{\@tocpagenum{#7}}\par
    \nobreak
    \endgroup
  \fi}
\makeatother
\makeatletter
\newcommand{\xdashrightarrow}[2][]{\ext@arrow 0359\rightarrowfill@@{#1}{#2}}
\def\rightarrowfill@@{\arrowfill@@\relax\relbar\rightarrow}
\def\arrowfill@@#1#2#3#4{%
  $\m@th\thickmuskip0mu\medmuskip\thickmuskip\thinmuskip\thickmuskip
   \relax#4#1
   \xleaders\hbox{$#4#2$}\hfill
   #3$%
}
\makeatother
\newcommand{\sfd}{{\sf d}}
\renewcommand{\d}{{\rm d}}
\newcommand{\N}{\mathbb{N}}
\newcommand{\R}{\mathbb{R}}
\newcommand{\LIP}{{\rm LIP}}
\newcommand{\Lip}{{\rm Lip}}
\newcommand{\lip}{{\rm lip}}
\newcommand{\lipa}{{\rm lip}_a}
\newcommand{\spt}{{\rm spt}}
\newcommand{\ud}{\underline\d}
\newcommand{\eps}{\varepsilon}  
\newcommand{\nchi}{{\raise.3ex\hbox{$\chi$}}}
\newcommand{\weakto}{\rightharpoonup}
\renewcommand{\P}{{\rm P}}
\newcommand{\restr}[1]{\lower3pt\hbox{$|_{#1}$}}
\newcommand{\X}{{\rm X}}
\newcommand{\Y}{{\rm Y}}
\newcommand{\mm}{{\mathfrak m}}
\newcommand{\ppi}{\boldsymbol{\pi}}
\newcommand{\e}{{\rm e}}
\newcommand{\limi}{\varliminf}
\newcommand{\lims}{\varlimsup}
\newcommand{\fr}{\penalty-20\null\hfill$\blacksquare$} 
\numberwithin{equation}{section}
\newtheorem{theorem}{Theorem}[section]
\newtheorem{corollary}[theorem]{Corollary}
\newtheorem{lemma}[theorem]{Lemma}
\newtheorem{proposition}[theorem]{Proposition}
\newtheorem{definition}[theorem]{Definition}

\newtheorem{remark}[theorem]{Remark}
\title{Infinitesimal Hilbertianity of weighted Riemannian manifolds}
\author[D.~Lu\v ci\' c]{Danka Lu\v ci\' c}
\address{SISSA, via Bonomea 265, 34136 Trieste - Italy}
\email{dlucic@sissa.it}
\author[E.~Pasqualetto]{Enrico Pasqualetto}
\address{SISSA, via Bonomea 265, 34136 Trieste - Italy}
\email{epasqual@sissa.it}

\begin{document}
\begin{abstract}
The main result of this paper is the following: any `weighted' Riemannian
manifold $(M,g,\mu)$ -- i.e.\ endowed with a generic non-negative Radon measure $\mu$
-- is `infinitesimally Hilbertian', which means that its associated
Sobolev space $W^{1,2}(M,g,\mu)$ is a Hilbert space.

We actually prove a stronger result: the abstract tangent
module (\`{a} la Gigli) associated to any weighted reversible Finsler
manifold $(M,F,\mu)$ can be isometrically embedded into the space of
all measurable sections of the tangent bundle of $M$ that are $2$-integrable
with respect to $\mu$.
\end{abstract}

\makeatletter

\date{\today} 



\maketitle
\tableofcontents
\section*{Introduction}
\noindent\textit{General overview.}
In the rapidly expanding theory of geometric analysis over metric measure
spaces $(\X,\sfd,\mm)$ a key role is played by the notion of Sobolev space
$W^{1,2}(\X,\sfd,\mm)$ that has been proposed in
\cite{Cheeger00} (see also \cite{Shanmugalingam00,AGS11_HeatFlow}).
In general, the space $W^{1,2}(\X,\sfd,\mm)$ has a Banach space structure
but is not necessarily a Hilbert space.
Those metric measure spaces $(\X,\sfd,\mm)$ whose associated Sobolev space
$W^{1,2}(\X,\sfd,\mm)$ is Hilbert are said to be \emph{infinitesimally Hilbertian};
cf.\ \cite{Gigli12}. This choice of terminology is due to the fact that such requirement
captures, in a sense, the property of being a `Hilbert-like' space at
arbitrarily small scales.
\smallskip

Infinitesimally Hilbertian spaces are particularly relevant in several
situations. For instance, in the framework of synthetic lower Ricci curvature
bounds -- in the sense of Lott-Villani \cite{Lott-Villani07} and Sturm
\cite{Sturm06I,Sturm06II}, known as $\sf CD$ \emph{condition} --
the infinitesimal Hilbertianity assumption has been used to single out
the `Riemannian' structures among the `Finslerian' ones, thus bringing
forth the well-established notion of $\sf RCD$ \emph{space} \cite{AGMR15,AGS14,Gigli12}.
We refer to the surveys \cite{Villani2016,Villani2017,Ambrosio18} 
for a detailed account of the vast literature concerning the
${\sf CD}/{\sf RCD}$ conditions.
\smallskip

The main purpose of the present paper is to prove that any geodesically
complete Riemannian manifold  $(M,g)$ is `universally infinitesimally Hilbertian',
meaning that 
\[(M,\sfd_g,\mu)\text{ is infinitesimally Hilbertian for any Radon measure }
\mu\geq 0\text{ on }M,\]
where $\sfd_g$ stands for the distance on $M$ induced by the Riemannian metric $g$.
This will be achieved as an immediate consequence of the following result:
given a geodesically complete, reversible Finsler manifold $(M,F)$ and a non-negative
Radon measure $\mu$ on $M$, it holds that the `abstract' tangent module $L^2_\mu(TM)$
associated to $(M,F,\mu)$ in the sense of Gigli \cite{G18_NonSmooth} can be isometrically embedded
into the `concrete' space of all $L^2(\mu)$-sections of the tangent bundle $TM$ of $M$.
\bigskip

\noindent\textit{Motivation and related works.}
Our interest in universally infinitesimally Hilbertian metric spaces is mainly
motivated by the study of metric-valued Sobolev maps, as we are going to describe.
Given a metric measure space $(\X,\sfd_\X,\mm)$ and a complete separable metric
space $(\Y,\sfd_\Y)$, one of the possible ways to define the space
${\rm S}^2(\X;\Y)$ of `weakly differentiable' maps from $\X$ to $\Y$
is via post-composition; cf.\ \cite{HKST15}. As shown in \cite[Theorem 3.3]{GigPasSou18},
any Sobolev map $u\in{\rm S}^2(\X;\Y)$ can be naturally associated with an
$L^0(\mm)$-linear and continuous operator
\[\d u:\,L^0_\mm(T\X)\longrightarrow\big(u^*L^0_\mu(T^*\Y)\big)^*,\]
where the finite Borel measure $\mu$ is defined as $\mu:=u_*(|Du|^2\mm)$;
the map $\d u$ is called \emph{differential}.
(We refer to \cite[Section 2]{GigPasSou18} for a brief summary of the terminology
used above.) We underline that the measure $\mu$ is not given a priori,
but it rather depends on the map $u$ itself in a non-trivial manner.
This implies that the target module of $\d u$ might possess a very complicated
structure. One of the reasons why we focus on universally infinitesimally
Hilbertian spaces $(\Y,\sfd_\Y)$ is that the cotangent module $L^0_\mu(T^*\Y)$ is a
Hilbert module regardless of the chosen measure $\mu$. In particular,
the target space $(u^*L^0_\mu(T^*\Y))^*$ of the differential $\d u$ is a Hilbert
module as well and can be canonically identified with $u^*L^0_\mu(T\Y)$. This allows for
more refined calculus tools and nicer functional-analytic properties,
cf.\ \cite{G18_NonSmooth} for the related discussion. Even more importantly,
to show that the abstract tangent module $L^0_\mu(T\Y)$ isometrically embeds
into some geometric space of sections would provide a `more concrete' representation
of the differential operator $\d u$.
\smallskip

The results contained in this paper have been already proved in \cite{GP16_Behaviour} 
for the particular case in which the Finsler manifold $(M,F)$ under consideration is
the Euclidean space $\R^n$ equipped with any norm $\|\cdot\|$. In fact, the structure
of our proofs follows along the path traced by \cite{GP16_Behaviour}.
We also mention that in the forthcoming paper \cite{DiMarGigPasSou18} it is proven that
locally ${\sf CAT}(\kappa)$ spaces are universally infinitesimally Hilbertian;
we recall that these are geodesic metric spaces whose sectional curvature is
(locally) bounded from above by $\kappa\in\R$ in the sense
of Alexandrov. The motivation behind such result is that it could be helpful
(if used in conjunction with the notion of differential operator for metric-valued
Sobolev maps discussed above) in order to study the regularity properties
of harmonic maps from finite-dimensional $\sf RCD$ spaces to ${\sf CAT}(0)$ spaces.
\bigskip

\noindent\textit{Outline of the work.}
In Section \ref{s:preliminaries} we briefly recall the basics
of Sobolev calculus on metric measure spaces and the language of $L^2$-normed
$L^\infty$-modules proposed by Gigli in \cite{G18_NonSmooth}. 
\smallskip

Section \ref{s:Finsler} is entirely devoted to Finsler geometry.
After a short introduction to few basic concepts, we will be concerned with the
approximation of Lipschitz functions by $C^1$-functions.
Our new contribution in this regard, namely Theorem \ref{thm:C1_approx_Lip},
constitutes a `more local' version of similar results that have been proved
in \cite{AFLR07,JS11,GJR13}.
\smallskip

The core of the paper is Section \ref{s:main}. In Proposition \ref{prop:density_energy_C1}
we exploit the above-mentioned approximation result to bridge the gap between
the abstract Sobolev space associated to a weighted Finsler manifold $(M,F,\mu)$
and the `true' differentials of functions in $C^1_c(M)$.
This represents the key passage to build a quotient projection map
from the space $\Gamma_2(T^*M;\mu)$ of all $L^2(\mu)$-sections of $T^*M$ to
$L^2_\mu(T^*M)$ (Lemma \ref{lem:def_P}, Proposition \ref{prop:P_quotient}).
We thus obtain -- by duality -- an isometric embedding of $L^2_\mu(TM)$ into the
space $\Gamma_2(TM;\mu)$ of all $L^2(\mu)$-sections of $TM$
(Theorem \ref{thm:iota_isometric}). As a direct corollary, any weighted Riemannian
manifold is infinitesimally Hilbertian (Theorem \ref{thm:UiH_Riemannian}).
\smallskip

Finally, in Section \ref{s:alternative_proof} we provide an alternative proof
of Theorem \ref{thm:UiH_Riemannian}, which does not rely upon Theorem
\ref{thm:iota_isometric}. This approach combines the analogue of Theorem
\ref{thm:UiH_Riemannian} for the Euclidean space proven in \cite{GP16_Behaviour}
with a localisation argument.
Nonetheless, we preferred to follow the first approach in order to place the emphasis
on Theorem \ref{thm:iota_isometric}, because of its independent interest.
\bigskip

\noindent\textbf{Acknowledgements.} The authors would like to acknowledge
Nicola Gigli and Martin Kell for the useful comments
and suggestions about this paper.
\section{Preliminaries on metric measure spaces}\label{s:preliminaries}
\subsection{Notation on metric spaces}
Consider a metric space $(\X,\sfd)$. Given any $x\in\X$ and $r>0$,
we denote by $B^\X_r(x)$ the open ball in $(\X,\sfd)$ with center
$x$ and radius $r$. More generally, we denote by $B^\X_r(E)$ the
$r$-neighbourhood of any set $E\subseteq\X$. We shall sometimes
work with metric spaces having the property that the closure of any
ball is compact: such spaces are said to be \emph{proper}.
\medskip

We shall use the notation $\LIP(\X)$ to indicate the family of all
real-valued Lipschitz functions defined on $\X$, while $\LIP_c(\X)$
will be the set of all functions in $\LIP(\X)$ having compact support.
Given any $f\in\LIP(\X)$, let us introduce the following quantities:
\begin{itemize}
\item[$\rm i)$] \textsc{Global Lipschitz constant.}
Let $E\subseteq\X$ be any set (containing at least two elements).
Then we denote by $\Lip(f;E)$ the Lipschitz constant of $f\restr E$, i.e.
\begin{equation}
\Lip(f;E):=\sup\Bigg\{\frac{\big|f(x)-f(y)\big|}{\sfd(x,y)}
\;\Bigg|\;x,y\in E,\,x\neq y\Bigg\}.
\end{equation}
For the sake of brevity, we shall write $\Lip(f)$ instead of $\Lip(f;\X)$.
\item[$\rm ii)$] \textsc{Local Lipschitz constant.}
We define the function $\lip(f):\,\X\to[0,+\infty)$ as
\begin{equation}
\lip(f)(x):=\underset{y\to x}\lims\,\frac{\big|f(x)-f(y)\big|}{\sfd(x,y)}
\quad\text{ for every accumulation point }x\in\X
\end{equation}
and $\lip(f)(x):=+\infty$ for every isolated point $x\in\X$.
\item[$\rm iii)$] \textsc{Asymptotic Lipschitz constant.} We define the
function $\lipa(f):\,\X\to[0,+\infty)$ as
\begin{equation}
\lipa(f)(x):=\inf_{r>0}\,\Lip\big(f;B^\X_r(x)\big)
\quad\text{ for every accumulation point }x\in\X
\end{equation}
and $\lipa(x):=+\infty$ for every isolated point $x\in\X$.
\end{itemize}
It can be readily checked that $\lip(f)\leq\lipa(f)\leq\Lip(f)$
is satisfied in $\X$.
\subsection{Sobolev calculus on metric measure spaces}
For our purposes, by \emph{metric measure space} we mean any
triple $(\X,\sfd,\mm)$, where
\begin{equation}\label{eq:def_mms}\begin{split}
(\X,\sfd)&\quad\text{ is a complete and separable metric space,}\\
\mm\neq 0&\quad\text{ is a non-negative Radon measure on }(\X,\sfd).
\end{split}\end{equation}
In order to introduce the notion of Sobolev space $W^{1,2}(\X,\sfd,\mm)$
proposed by L.\ Ambrosio, N.\ Gigli and G.\ Savar\'{e} in \cite{AGS11_HeatFlow},
we need to fix some notation. We say that a continuous curve $\gamma:\,[0,1]\to\X$
is \emph{absolutely continuous} provided there exists $f\in L^1(0,1)$ such that
$\sfd(\gamma_t,\gamma_s)\leq\int_s^t f(r)\,\d r$ holds for every $t,s\in [0,1]$
with $s<t$. The minimal $1$-integrable function (in the a.e.\ sense) that can be
chosen as $f$ is called \emph{metric speed} of $\gamma$ and denoted by $|\dot\gamma|$.
As proven in \cite[Theorem 1.1.2]{AGS14_GradFlow}, it holds that
$|\dot\gamma_t|=\lim_{h\to 0}\sfd(\gamma_{t+h},\gamma_t)/|h|$
for almost every $t\in(0,1)$.
\medskip

A \emph{test plan} on $\X$ is any Borel probability measure $\ppi$
on $C([0,1],\X)$ with the following properties:
\begin{itemize}
\item[$\rm i)$] There exists a constant $C>0$ such that $(\e_t)_*\ppi\leq C\mm$
holds for every $t\in[0,1]$, where the evaluation map $\e_t:\,C([0,1],\X)\to\X$
is given by $\e_t(\gamma):=\gamma_t$ and $(\e_t)_*\ppi$ stands for the
pushforward measure of $\ppi$ under $\e_t$.
\item[$\rm ii)$] It holds that
$\int\!\!\int_0^1|\dot\gamma_t|^2\,\d t\,\d\ppi(\gamma)<+\infty$, with the
convention that $\int_0^1|\dot\gamma_t|^2\,\d t:=+\infty$ when the curve
$\gamma$ is not absolutely continuous.
\end{itemize}
In particular, any test plan is concentrated on the family
of all absolutely continuous curves on $\X$.
\begin{definition}[Sobolev space \cite{AGS11_HeatFlow}]\label{def:Sobolev_space}
We define the \emph{Sobolev space} $W^{1,2}(\X,\sfd,\mm)$ as the set
of all functions $f\in L^2(\mm)$ with the following property: there
exists $G\in L^2(\mm)$ such that
\begin{equation}\label{eq:ineq_Sobolev_space}
\int\big|f(\gamma_1)-f(\gamma_0)\big|\,\d\ppi(\gamma)\leq
\int\!\!\!\int_0^1 G(\gamma_t)\,|\dot\gamma_t|\,\d t\,\d\ppi(\gamma)
\quad\text{ for every test plan }\ppi\text{ on }\X.
\end{equation}
Any such function $G$ is said to be a \emph{weak upper gradient} of $f$.
The minimal weak upper gradient of the function $f$ -- intended in the
$\mm$-a.e.\ sense -- is denoted by $|Df|$.
\end{definition}

The Sobolev space $W^{1,2}(\X,\sfd,\mm)$ is a Banach space if endowed with the norm
\begin{equation}\label{eq:norm_Sobolev_space}
{\|f\|}_{W^{1,2}(\X,\sfd,\mm)}:=\Big({\|f\|}^2_{L^2(\mm)}
+{\big\||Df|\big\|}^2_{L^2(\mm)}\Big)^{1/2}
\quad\text{ for every }f\in W^{1,2}(\X,\sfd,\mm),
\end{equation}
but in general it is not a Hilbert space. For this reason, the following definition
is meaningful:
\begin{definition}[Infinitesimal Hilbertianity]\label{def:iH}
We say that the metric measure space $(\X,\sfd,\mm)$ is \emph{infinitesimally Hilbertian}
provided its associated Sobolev space $W^{1,2}(\X,\sfd,\mm)$ is a Hilbert space.
\end{definition}

An important property of minimal weak upper gradients is
their lower semicontinuity (cf.\ \cite{AGS11_HeatFlow}):
\begin{proposition}\label{prop:lsc_mwug}
Let $(f_n)_{n\in\N}\subseteq W^{1,2}(\X,\sfd,\mm)$ satisfy $f_n\to f$ in $L^2(\mm)$
for some $f\in L^2(\mm)$. Suppose also that $|Df_n|\weakto G$ weakly in $L^2(\mm)$
for some $G\in L^2(\mm)$. Then $f\in W^{1,2}(\X,\sfd,\mm)$ and the inequality
$|Df|\leq G$ holds $\mm$-a.e.\ in $\X$.
\end{proposition}

We point out that $W^{1,2}(\X,\sfd,\mm)$ contains all Lipschitz functions on $\X$
having compact support. More precisely, given any function $f\in\LIP_c(\X)$ it holds that
\begin{equation}\label{eq:|Df|_leq_lip(f)}
|Df|\leq\lip(f)\quad\text{ in the }\mm\text{-a.e.\ sense.}
\end{equation}
On proper spaces, Lipschitz functions with compact support
are \emph{dense in energy} in $W^{1,2}(\X,\sfd,\mm)$:
\begin{theorem}[Ambrosio-Gigli-Savar\'{e} \cite{AGS11_DensityLip}]\label{thm:density_energy_Lip}
Suppose $(\X,\sfd,\mm)$ is a proper metric measure space.
Fix any Sobolev function $f\in W^{1,2}(\X,\sfd,\mm)$. Then there
exists a sequence $(f_n)_{n\in\N}\subseteq\LIP_c(\X)$
such that $f_n\to f$ and $\lipa(f_n)\to|Df|$ in $L^2(\mm)$
as $n\to\infty$.
\end{theorem}
\subsection{Abstract tangent and cotangent modules}
Consider a metric measure space $(\X,\sfd,\mm)$. We assume the reader to
be familiar with the language of \emph{$L^2(\mm)$-normed $L^\infty(\mm)$-modules},
which has been developed in the papers \cite{G18_NonSmooth, G17_Lecture}.
\medskip

We just recall that there is a unique couple $\big(L^2_\mm(T^*\X),\d\big)$
-- where $L^2_\mm(T^*\X)$ is an $L^2(\mm)$-normed $L^\infty(\mm)$-module called
\emph{cotangent module} and $\d:\,W^{1,2}(\X,\sfd,\mm)\to L^2_\mm(T^*\X)$ is a
linear operator called \emph{differential} --
such that the following two conditions are satisfied:
\begin{itemize}
\item[$\rm i)$] It holds that $|\d f|=|Df|$ in the $\mm$-a.e.\ sense
for every $f\in W^{1,2}(\X,\sfd,\mm)$.
\item[$\rm ii)$] The set $\big\{\d f\,:\,f\in W^{1,2}(\X,\sfd,\mm)\big\}$
generates $L^2_\mm(T^*\X)$ in the sense of modules.
\end{itemize}
The module dual of $L^2_\mm(T^*\X)$ is called \emph{tangent module} and
denoted by $L^2_\mm(T\X)$.
\medskip

A fundamental property of the differential -- which follows from Proposition
\ref{prop:lsc_mwug} -- is that it is a closed operator;
cf.\ \cite[Theorem 2.2.9]{G18_NonSmooth}:
\begin{proposition}[Closure of $\d$]\label{prop:closure_diff}
Let $(f_n)_{n\in\N}\subseteq W^{1,2}(\X,\sfd,\mm)$ be a sequence satisfying
\begin{equation}\label{eq:closure_d}\begin{split}
f_n\weakto f&\quad\text{ weakly in }L^2(\mm),\\
\d f_n\weakto\omega&\quad\text{ weakly in }L^2_\mm(T^*\X),
\end{split}\end{equation}
for some $f\in L^2(\mm)$ and $\omega\in L^2_\mm(T^*\X)$.
Then $f\in W^{1,2}(\X,\sfd,\mm)$ and $\d f=\omega$.
\end{proposition}
Furthermore, the following result is taken from \cite[Proposition 2.2.10]{G18_NonSmooth}:
\begin{proposition}[Reflexivity of the Sobolev space]\label{prop:weak_cptness_d}
The following conditions are equivalent:
\begin{itemize}
\item[$\rm i)$] The Sobolev space $W^{1,2}(\X,\sfd,\mm)$ is reflexive.
\item[$\rm ii)$] Given any bounded sequence $(f_n)_{n\in\N}\subseteq W^{1,2}(\X,\sfd,\mm)$,
there exist $f\in W^{1,2}(\X,\sfd,\mm)$ and a subsequence
$(f_{n_k})_{k\in\N}$ such that $(f_{n_k},\d f_{n_k})\weakto(f,\d f)$
weakly in $L^2(\mm)\times L^2_\mm(T^*\X)$.
\end{itemize}
In particular, if $L^2_\mm(T^*\X)$ is reflexive then $W^{1,2}(\X,\sfd,\mm)$ is reflexive.
\end{proposition}

Finally, we point out that
\begin{equation}\label{eq:characterisation_W12_Sobolev}
W^{1,2}(\X,\sfd,\mm)\text{ is a Hilbert space}
\quad\Longleftrightarrow\quad L^2_\mm(T\X)\text{ is a Hilbert module,}
\end{equation}
as proven in \cite[Proposition 2.3.17]{G18_NonSmooth}.
\section{Some properties of Finsler manifolds}\label{s:Finsler}
\subsection{Definition and basic results}
For our purposes, by \emph{manifold} we shall always mean a connected
differentiable manifold of class $C^\infty$. Given a manifold $M$ and
a point $x\in M$, we denote by $T_x M$ the tangent space of $M$ at $x$
and by $\exp_x$ the exponential map at $x$. We make use of the notation
$TM=\bigsqcup_{x\in M}T_x M$ to indicate the tangent bundle of $M$.
Moreover, we denote by $T^*_x M$ and $T^*M$ the cotangent space of $M$ at $x$
and the cotangent bundle of $M$, respectively.
We now briefly report the definition of Finsler structure
over a manifold, referring to the monograph \cite{BCS12} for a thorough account about this topic.
\medskip

Let $V$ be a given finite-dimensional vector space over $\R$. Then a \emph{Minkowski norm}
on $V$ is any functional $F:\,V\to[0,+\infty)$ satisfying the following properties:
\begin{itemize}
\item[$\rm i)$] \textsc{Positive definiteness.} Given any $v\in V$,
we have that $F(v)=0$ if and only if $v=0$.
\item[$\rm ii)$] \textsc{Triangle inequality.} It holds that $F(v+w)\leq F(v)+F(w)$
for every $v,w\in V$.
\item[$\rm iii)$] \textsc{Positive homogeneity.} We have that $F(\lambda v)=\lambda F(v)$
for every $v\in V$ and $\lambda\geq 0$.
\item[$\rm iv)$] \textsc{Regularity.} The function $F$ is continuous on $V$
and of class $C^\infty$ on $V\setminus\{0\}$.
\item[$\rm v)$] \textsc{Strong convexity.} Given any $v\in V\setminus\{0\}$,
it holds that the quadratic form
\begin{equation}\label{eq:quadratic_form_F}
V\ni w\longmapsto\frac{1}{2}\,\d^2(F^2)_v[w,w]
\end{equation}
is positive definite. (The expression in \eqref{eq:quadratic_form_F} stands
for the second differential of $F^2$ at $v$.)
\end{itemize}
In particular, any Minkowski norm is an asymmetric norm.
\begin{definition}[Finsler manifold]\label{def:Finsler_manifold}
A \emph{Finsler manifold} is any couple $(M,F)$, where $M$ is a
given manifold and $F:\,TM\to[0,+\infty)$ is a continuous function
satisfying the following properties:
\begin{itemize}
\item[$\rm i)$] The function $F$ is of class $C^\infty$ on $TM\setminus\{0\}$.
\item[$\rm ii)$] The functional $F(x,\cdot):\,T_x M\to[0,+\infty)$ is
a Minkowski norm for every $x\in M$.
\end{itemize}
Moreover, we say that $(M,F)$ is \emph{reversible} provided each
function $F(x,\cdot)$ is symmetric, i.e.
\begin{equation}\label{eq:reversible_Finsler}
F(x,-v)=F(x,v)\quad\text{ for every }x\in M\text{ and }v\in T_x M.
\end{equation}
Condition \eqref{eq:reversible_Finsler} is equivalent to requiring that
each $F(x,\cdot)$ is a (symmetric) norm on $T_x M$.
\end{definition}

We point out that any Riemannian manifold is a special case of
reversible Finsler manifold.
(This is an abuse of notation. More precisely: if $(M,g)$ is a Riemannian
manifold, then $(M,F)$ is a reversible Finsler manifold,
where we set $F(x,v):=g_x(v,v)^{1/2}$ for every $x\in M$ and $v\in T_x M$.)
\begin{definition}[Finsler distance]\label{def:Finsler_distance}
Let $(M,F)$ be a reversible Finsler manifold. Given any piecewise
$C^1$ curve $\gamma:\,[0,1]\to M$, we define its \emph{Finsler length} as
\begin{equation}\label{eq:def_Finsler_length}
\ell_F(\gamma):=\int_0^1 F(\gamma_t,\dot\gamma_t)\,\d t.
\end{equation}
Then we define the \emph{Finsler distance} $\sfd_F(x,y)$ between two points $x,y\in M$ as
\begin{equation}\label{eq:def_Finsler_distance}
\sfd_F(x,y):=\inf\Big\{\ell_F(\gamma)\;\Big|\;\gamma:\,[0,1]\to M\text{ piecewise }C^1
\text{ with }\gamma_0=x\text{ and }\gamma_1=y\Big\}.
\end{equation}
A \emph{Finsler geodesic} is any $C^1$-curve on $M$ that is locally
a stationary point of the length functional.
\end{definition}
\begin{remark}{\rm
When $(M,F)$ is a (not reversible) Finsler manifold, one has that the
formula \eqref{eq:def_Finsler_distance} defines a quasi-distance on $M$
rather than a distance in the usual sense. Our main approximation result
-- namely Theorem \ref{thm:C1_approx_Lip} below -- still holds true even in the
case of general Finsler manifolds (this can be achieved with minor modifications
of the arguments that we shall see). Nevertheless, we prefer to focus our
attention on the reversible case, the reason being that the language of Sobolev
calculus and (co)tangent modules is so far available just for metric structures.
\fr}\end{remark}

For a proof of the ensuing result in the Finsler case,
we refer e.g.\ to \cite[Theorem 6.6.1]{BCS12}.
\begin{theorem}[Hopf-Rinow]\label{thm:Hopf-Rinow}
Let $(M,F)$ be a reversible Finsler manifold.
Then the following four conditions are equivalent:
\begin{itemize}
\item[$\rm i)$] The Finsler manifold $(M,F)$ is \emph{geodesically complete},
i.e.\ any constant speed geodesic can be extended to a geodesic defined
on the whole real line.
\item[$\rm ii)$] The metric space $(M,\sfd_F)$ is complete.
\item[$\rm iii)$] Given any $x\in M$, it holds that the exponential map $\exp_x$
is defined on the whole $T_x M$.
\item[$\rm iv)$] The metric space $(M,\sfd_F)$ is proper.
\end{itemize}
\end{theorem}
\subsection{Smooth approximation of Lipschitz functions}
In the sequel, we shall need the following result concerning the biLipschitz
behaviour of the exponential map on sufficiently small balls:
\begin{theorem}[Deng-Hou \cite{Deng02}]\label{thm:Deng-Hou}
Let $(M,F)$ be a reversible Finsler manifold. Fix a point $x\in M$
and some constant $\eps>0$. Then there exists a radius $r>0$ such
that the exponential map
\begin{equation}\label{eq:exp_diffeo}
\exp_x:\,B^{T_x M}_r(0)\longrightarrow B^M_r(x)
\end{equation}
is a $(1+\eps)$-biLipschitz $C^1$-diffeomorphism.
\end{theorem}

We now present a new result about regularisation of Lipschitz functions
on a reversible Finsler manifold $(M,F)$. Roughly speaking, it states that
any Lipschitz function $f:\,M\to\R$ can be uniformly approximated by
functions of class $C^1$ whose Lipschitz constant is locally controlled
by that of $f$. This represents a `local' variant of the approximation
theorem proven in \cite{GJR13}.
\begin{theorem}\label{thm:C1_approx_Lip}
Let $(M,F)$ be a geodesically complete, reversible Finsler manifold.
Fix a Lipschitz function $f\in\LIP(M)$ and some constants
$\delta,\eps,\lambda>0$. Then there exists a function $g\in C^1(M)$ with
$\spt(g)\subseteq B^M_\delta\big(\spt(f)\big)$ such that
\begin{equation}\label{eq:claim_C1_approx_Lip}\begin{split}
\big|g(x)-f(x)\big|&\leq\eps\\
\lipa(g)(x)&\leq\Lip\big(f;B^M_\delta(x)\big)+\lambda
\end{split}
\quad\text{ for every }x\in M.
\end{equation}
\end{theorem}
\begin{proof} We divide the proof into several steps:
\smallskip

\noindent{\color{blue}\textsc{Step 1: Set-up.}} Fix any $r>0$ such that $r\leq\delta/2$ and
\begin{equation}\label{eq:condition_on_r}
(2r+r^2)\,\Lip(f)+r\leq\lambda.
\end{equation}
Theorem \ref{thm:Deng-Hou} grants that for any $x\in M$ we can
pick a radius $r_x\in(0,r)$ such that the exponential map
$\exp_x:\,B^{T_x M}_{r_x}(0)\to B^M_{r_x}(x)$ is a
$(1+r)$-biLipschitz $C^1$-diffeomorphism. Hence we can choose a
sequence $(x_i)_{i\in\N}\subseteq M$ such that the
family $(B_i)_{i\in\N}$ -- where we set $B_i:=B^M_{r_{x_i}}(x_i)$
for all $i\in\N$ -- is a locally finite open covering of $M$.
Given any $i\in\N$, we fix a linear isomorphism $I_i:\,\R^n\to T_{x_i}M$,
where $n:=\dim(M)$. Let us define the norm $\|\cdot\|_i$ on $\R^n$ as
\begin{equation}\label{eq:def_norm_i}
\|v\|_i:=F\big(x_i,I_i(v)\big)\quad\text{ for every }v\in\R^n.
\end{equation}
Since any two norms on $\R^n$ are equivalent, there exists $C_i\geq 1$ such that
\begin{equation}\label{eq:norm_i_equiv}
\frac{1}{C_i}\,\|v\|_i\leq|v|\leq C_i\,\|v\|_i\quad\text{ for every }v\in\R^n.
\end{equation}
We define the chart $\varphi_i:\,B_i\to\R^n$ as
\begin{equation}\label{eq:def_chart_phi_i}
\varphi_i(x):=(\exp_{x_i}\circ I_i)^{-1}(x)\quad\text{ for every }x\in B_i.
\end{equation}
Therefore $\varphi_i$ is a $(1+r)$-biLipschitz $C^1$-diffeomorphism from
$(B_i,\sfd_F)$ to $\big(\varphi_i(B_i),\|\cdot\|_i\big)$. Moreover, let us fix
a smooth partition of unity $(\psi_i)_{i\in\N}$ subordinated to the covering
$(B_i)_{i\in\N}$, i.e.
\begin{itemize}
\item the functions $\psi_i$ belong to $C^\infty_c(M)$,
\item $0\leq\psi_i\leq 1$ and $\spt(\psi_i)\subseteq B_i$ for every $i\in\N$,
\item $\sum_{i\in\N}\psi_i(x)=1$ holds for every $x\in M$.
\end{itemize}
Finally, for any $i\in\N$ we call $\mathcal A_i:=\{j\in\N\,:\,B_j\cap B_i\neq\emptyset\}$,
we denote by $n_i\in\N$ the cardinality of the set $\mathcal A_i$ and we
define $m_i:=\max\{n_j\,:\,j\in\mathcal A_i\}\in\N$. Then it is immediate to check that
\begin{equation}\label{eq:property_m_i}
n_i\leq m_j\quad\text{ for every }i\in\N\text{ and }j\in\mathcal A_i.
\end{equation}
{\color{blue}\textsc{Step 2: Construction of $g$.}} First of all, fix a
family $(\rho_k)_{k\in\N}$ of smooth mollifiers on $\R^n$, i.e.
\begin{itemize}
\item the functions $\rho_k$ are symmetric and belong to $C^\infty_c(\R^n)$,
\item $\rho_k\geq 0$ and $\spt(\rho_k)\subseteq B^{\R^n}_{1/k}(0)$ for every $k\in\N$,
\item $\int_{\R^n}\rho_k(v)\,\d v=1$ holds for every $k\in\N$.
\end{itemize}
For any $i\in\N$ we can choose a McShane extension
$f_i:\,\big(\R^n,\|\cdot\|_i\big)\to\R$ of $f\circ\varphi_i^{-1}:\,\varphi_i(B_i)\to\R$,
namely $f_i$ is a Lipschitz function with $\Lip(f_i)\leq(1+r)\Lip(f;B_i)$ that coincides
with $f\circ\varphi_i^{-1}$ on the set $\varphi_i(B_i)$. Now we define
$f^k_i:\,\R^n\to\R$ for any $i,k\in\N$ as
\begin{equation}\label{eq:def_f^k_i}
f^k_i(v):=(f_i*\rho_k)(v)=\int_{\R^n}f_i(v+w)\rho_k(w)\,\d w
\quad\text{ for every }v\in\R^n.
\end{equation}
It is well-known that each function $f^k_i$ is of class $C^\infty$.
Pick a sequence $(k_i)_{i\in\N}\subseteq\N$ for which
\begin{equation}\label{eq:def_k_i}\begin{split}
\frac{(1+r)\,\Lip(f;B_i)\,C_i}{k_i}&\leq\eps,\\
\frac{\Lip(\psi_i)\,(1+r)\,\Lip(f;B_i)\,C_i}{k_i}&\leq\frac{r}{m_i}
\end{split}
\quad\text{ for every }i\in\N.
\end{equation}
Then we define $g_i:=f^{k_i}_i$ for all $i\in\N$ and
\begin{equation}\label{eq:def_g}
g(x):=\sum_{i\in\N}\psi_i(x)(g_i\circ\varphi_i)(x)\quad\text{ for every }x\in M.
\end{equation}
It clearly turns out that $g$ belongs to the space $C^1(M)$.
\smallskip

\noindent{\color{blue}\textsc{Step 3: Properties of $g$.}} Given $i\in\N$ and $v\in\R^n$,
it holds that
\begin{equation}\label{eq:estimate_g_i-f_i}\begin{split}
\big|g_i(v)-f_i(v)\big|&\overset{\phantom{\eqref{eq:norm_i_equiv}}}=
\bigg|\int_{\R^n}f_i(v+w)\rho_{k_i}(w)\,\d w
-\int_{\R^n}f_i(v)\rho_{k_i}(w)\,\d w\bigg|\\
&\overset{\phantom{\eqref{eq:norm_i_equiv}}}\leq
\int_{\R^n}\big|f_i(v+w)-f_i(v)\big|\rho_{k_i}(w)\,\d w
\leq\Lip(f_i)\int_{B^{\R^n}_{1/{k_i}}(0)}\|w\|_i\,\rho_{k_i}(w)\,\d w\\
&\overset{\eqref{eq:norm_i_equiv}}\leq
\frac{(1+r)\,\Lip(f;B_i)\,C_i}{k_i}\int_{\R^n}\rho_{k_i}(w)\,\d w
=\frac{(1+r)\,\Lip(f;B_i)\,C_i}{k_i}\overset{\eqref{eq:def_k_i}}\leq\eps,
\end{split}\end{equation}
thus accordingly one has that
\[
\big|g(x)-f(x)\big|\overset{\eqref{eq:def_g}}=
\bigg|\sum_{i\in\N}\psi_i(x)(g_i\circ\varphi_i-f)(x)\bigg|
\leq\sum_{i\in\N}\psi_i(x)\big|g_i-f\circ\varphi_i^{-1}\big|\big(\varphi_i(x)\big)
\overset{\eqref{eq:estimate_g_i-f_i}}\leq\eps\sum_{i\in\N}\psi_i(x)=\eps,
\]
which proves the first line of \eqref{eq:claim_C1_approx_Lip}.
Moreover, calling $S$ the set of all $i\in\N$ such that the center of the ball $B_i$
does not lie in $B^M_r\big(\spt(f)\big)$, we have for any $i\in S$ that
\[
f\restr{B_i}\equiv 0\quad\Longrightarrow\quad
f_i\equiv 0\quad\Longrightarrow\quad g_i\equiv 0,
\]
whence accordingly $g=\sum_{i\in\N\setminus S}\psi_i\,g_i\circ\varphi_i$. This shows that
\[
\spt(g)\subseteq\bigcup_{i\in\N\setminus S}B_i\subseteq B^M_{2r}\big(\spt(f)\big)
\subseteq B^M_\delta\big(\spt(f)\big).
\]
{\color{blue}\textsc{Step 4: Properties of $\lipa(g)$.}} Given $i\in\N$ and $v,w\in\R^n$,
it holds that
\begin{equation}\label{eq:estimate_Lip(g_i)}
\big|g_i(v)-g_i(w)\big|\leq\int_{\R^n}\big|f_i(v+u)-f_i(w+u)\big|\,\rho_{k_i}(u)\,\d u
\leq(1+r)\,\Lip(f;B_i)\,\|v-w\|_i.
\end{equation}
Now fix $x\in M$ and denote $\mathcal I_x:=\{i\in\N\,:\,x\in B_i\}$.
Pick any $i\in\mathcal I_x$ and notice that $\mathcal I_x\subseteq\mathcal A_i$.
Since the set $\mathcal I_x$ is finite, we can choose a radius $s_x>0$ satisfying
$B^M_{s_x}(x)\subseteq B_j$ for all $j\in\mathcal I_x$. Hence for every
$y,z\in B^M_{s_x}(x)$ one has that
\[\begin{split}
\big|g(y)-g(z)\big|&\leq
\bigg|\sum_{j\in\mathcal I_x}\big[\psi_j(y)-\psi_j(z)\big](g_j\circ\varphi_j-f)(y)\bigg|
+\bigg|\sum_{j\in\mathcal I_x}\psi_j(z)\big[(g_j\circ\varphi_j)(y)
-(g_j\circ\varphi_j)(z)\big]\bigg|\\
&\leq\underset{=:({\rm A})}
{\underbrace{\sum_{j\in\mathcal I_x}\big|\psi_j(y)-\psi_j(z)\big|\,
\big|(g_j\circ\varphi_j-f)(y)\big|}}+
\underset{=:({\rm B})}{\underbrace{\sum_{j\in\mathcal I_x}
\psi_j(z)\,\big|(g_j\circ\varphi_j)(y)-(g_j\circ\varphi_j)(z)\big|}}.
\end{split}\]
We separately estimate the quantities (A) and (B). Firstly, observe that
\[\begin{split}
({\rm A})&\overset{\eqref{eq:estimate_g_i-f_i}}\leq\sfd_F(y,z)\sum_{j\in\mathcal I_x}
\Lip(\psi_j)\frac{(1+r)\,\Lip(f;B_j)\,C_j}{k_j}\\
&\overset{\phantom{\eqref{eq:estimate_g_i-f_i}}}\leq\sfd_F(y,z)\sum_{j\in\mathcal A_i}
\frac{\Lip(\psi_j)\,(1+r)\,\Lip(f;B_j)\,C_j}{k_j}
\overset{\eqref{eq:def_k_i}}\leq\sfd_F(y,z)\sum_{j\in\mathcal A_i}\frac{r}{m_j}
\overset{\eqref{eq:property_m_i}}\leq r\,\sfd_F(y,z).
\end{split}\]
Furthermore, we have that
\[\begin{split}
({\rm B})&\overset{\eqref{eq:estimate_Lip(g_i)}}\leq
(1+r)\sum_{j\in\mathcal I_x}\psi_j(z)\,\Lip(f;B_j)\big\|\varphi_j(y)-\varphi_j(z)\big\|_j
\leq(1+r)^2\,\sfd_F(y,z)\sum_{j\in\mathcal I_x}\psi_j(z)\,\Lip(f;B_j)\\
&\overset{\phantom{\eqref{eq:estimate_Lip(g_i)}}}\leq
(1+r)^2\,\sfd_F(y,z)\,\Lip\big(f;B^M_{2r}(x)\big)\sum_{j\in\mathcal I_x}\psi_j(z)\\
&\overset{\phantom{\eqref{eq:estimate_Lip(g_i)}}}\leq
\Big[\Lip\big(f;B^M_\delta(x)\big)+(2r+r^2)\,\Lip(f)\Big]\sfd_F(y,z).
\end{split}\]
Therefore we finally conclude that for any $y,z\in B^M_{s_x}(x)$ it holds that
\[\begin{split}
\big|g(y)-g(z)\big|\leq
\Big[\Lip\big(f;B^M_\delta(x)\big)+(2r+r^2)\Lip(f)+r\Big]\sfd_F(y,z)
\overset{\eqref{eq:condition_on_r}}\leq
\Big[\Lip\big(f;B^M_\delta(x)\big)+\lambda\Big]\sfd_F(y,z).
\end{split}\]
This shows that
$\lipa(g)(x)\leq\Lip\big(g;B^M_{s_x}(x)\big)\leq\Lip\big(f;B^M_\delta(x)\big)+\lambda$
for every $x\in M$, thus proving the second line in \eqref{eq:claim_C1_approx_Lip}.
Hence the statement is achieved.
\end{proof}
\begin{remark}{\rm
On general Finsler manifolds the exponential map is only of class $C^1$.
Moreover -- as proven by Akbar-Zadeh in \cite{AZ88} -- the exponential map is of
class $C^2$ if and only if it is smooth. The family of those Finsler manifolds
having this property (that are said to be \emph{of Berwald type}) contains
all Riemannian manifolds. We observe that if $(M,F)$ is of Berwald type, then
the approximating function $g$ in Theorem \ref{thm:C1_approx_Lip} can be
chosen to be smooth (by the same proof).
\fr}\end{remark}
\section{Main result}\label{s:main}
Let us consider a geodesically complete, reversible Finsler manifold $(M,F)$
and a non-negative Radon measure $\mu$ on the metric space $(M,\sfd_F)$,
which will remain fixed for the whole section.
\begin{remark}{\rm
Observe that
\begin{equation}\label{eq:M_mms}
(M,\sfd_F,\mu)\text{ is a metric measure space, in the sense of }\eqref{eq:def_mms}.
\end{equation}
Indeed, the metric space $(M,\sfd_F)$ is complete (by Theorem \ref{thm:Hopf-Rinow})
and separable (as the manifold $M$ is second-countable by definition).
\fr}\end{remark}
\subsection{Density in energy of \texorpdfstring{$C^1$}{C1} functions}
Let $f\in C^1_c(M)$ be given. Then we denote by $\ud f$ its differential,
which is a continuous section of the cotangent bundle $T^*M$.
For brevity, let us set
\begin{equation}\label{eq:def_pointwise_norm_differential}
|\ud f|(x):=F^*\big(x,\ud f(x)\big)\quad\text{ for every }x\in M,
\end{equation}
where $F^*(x,\cdot)$ stands for the dual norm of $F(x,\cdot)$.
Observe that the function $f$ can be viewed as an element of the Sobolev
space $W^{1,2}(M,\sfd_F,\mu)$ and that
\begin{equation}\label{eq:|df|_leq_|udf|}
|\d f|\leq|\ud f|\quad\text{ in the }\mu\text{-a.e.\ sense,}
\end{equation}
as a consequence of \eqref{eq:|Df|_leq_lip(f)} and the fact that $\lip(f)=|\ud f|$.
\begin{proposition}[Density in energy of $C^1$ functions]\label{prop:density_energy_C1}
Let $f\in W^{1,2}(M,\sfd_F,\mu)$ be given. Then there exists a sequence
$(f_k)_{k\in\N}\subseteq C^1_c(M)$ such that $f_k\to f$ and $|\ud f_k|\to|\d f|$
in $L^2(\mu)$ as $k\to\infty$.
\end{proposition}
\begin{proof}
First of all, we know from Theorem \ref{thm:density_energy_Lip} that
there exists a sequence $(g_k)_{k\in\N}\subseteq\LIP_c(M)$ such that
$g_k\to f$ and $\lipa(g_k)\to|\d f|$ in $L^2(\mu)$. (Recall that
$(M,\sfd_F)$ is proper by Theorem \ref{thm:Hopf-Rinow}.) 
Now fix $k\in\N$ and observe that Theorem \ref{thm:C1_approx_Lip}
provides us with a sequence $(g_k^i)_{i\in\N}\subseteq C^1_c(M)$ with
\begin{equation}\label{eq:properties_g_k^i}\begin{split}
\big|g_k^i(x)-g_k(x)\big|&\leq\frac{1}{i},\\
\spt(g_k^i)&\subseteq B^M_{1/i}\big(\spt(g_k)\big),\\
\lipa(g_k^i)(x)&\leq\Lip\big(g_k;B^M_{1/i}(x)\big)+\frac{1}{i}
\end{split}
\quad\text{ for every }i\in\N\text{ and }x\in M.
\end{equation}
Notice that the first two lines in \eqref{eq:properties_g_k^i} yield
$\lim_i{\|g_k^i-g_k\|}_{L^2(\mu)}\leq\lim_i\mu\big(B^M_1(\spt(g_k))\big)^{1/2}/i=0$,
while the third one grants that
$\lims_i|\ud g_k^i|(x)=\lims_i\lipa(g_k^i)(x)\leq\lipa(g_k)(x)$
for every $x\in M$. Since we also have that
$|\ud g_k^i|\leq\nchi_{B^M_1(\spt(g_k))}\big(\Lip(g_k)+1\big)\in L^2(\mu)$
for all $i\in\N$, it follows from the reverse Fatou lemma that
$\lims_i{\big\||\ud g_k^i|\big\|}_{L^2(\mu)}\leq{\big\|\lipa(g_k)\big\|}_{L^2(\mu)}$.
Therefore a diagonal argument gives us a sequence $(i_k)_{k\in\N}\subseteq\N$
such that the functions $f_k:=g_k^{i_k}$ satisfy $f_k\to f$ in $L^2(\mu)$ and
\begin{equation}\label{eq:lims|udf_k|_leq_|df|}
\lims_{k\to\infty}{\big\||\ud f_k|\big\|}_{L^2(\mu)}\leq{\big\||\d f|\big\|}_{L^2(\mu)}.
\end{equation}
In particular, both the sequences $\big(|\ud f_k|\big)_k$ and $\big(|\d f_k|\big)_k$
are bounded in $L^2(\mu)$ by \eqref{eq:lims|udf_k|_leq_|df|} and \eqref{eq:|df|_leq_|udf|},
thus (up to subsequence) it holds that $|\d f_k|\weakto h$ and $|\ud f_k|\weakto h'$
weakly for some $h,h'\in L^2(\mu)$. Then Proposition \ref{prop:lsc_mwug}
grants that $|\d f|\leq h\leq h'$ holds $\mu$-a.e.\ in $M$. Given that in
\[
{\big\||\d f|\big\|}_{L^2(\mu)}\leq{\|h'\|}_{L^2(\mu)}
\leq\limi_{k\to\infty}{\big\||\ud f_k|\big\|}_{L^2(\mu)}
\leq\lims_{k\to\infty}{\big\||\ud f_k|\big\|}_{L^2(\mu)}
\overset{\eqref{eq:lims|udf_k|_leq_|df|}}\leq{\big\||\d f|\big\|}_{L^2(\mu)}
\]
all inequalities are actually equalities, it holds that
${\|h'\|}_{L^2(\mu)}={\big\||\d f|\big\|}_{L^2(\mu)}
=\lim_k{\big\||\ud f_k|\big\|}_{L^2(\mu)}$. Hence we conclude that
$h'=|\d f|$ in the $\mu$-a.e.\ sense and accordingly $|\ud f_k|\to|\d f|$
in $L^2(\mu)$.
\end{proof}
\subsection{Concrete tangent and cotangent modules}
We define the ``concrete'' tangent/cotangent modules
associated to $(M,\sfd_F,\mu)$ as
\begin{equation}\label{eq:def_concrete_modules}\begin{split}
\Gamma_2(TM;\mu)&:=\text{space of all }L^2(\mu)\text{-sections of }TM,\\
\Gamma_2(T^*M;\mu)&:=\text{space of all }L^2(\mu)\text{-sections of }T^*M.
\end{split}\end{equation}
The space $\Gamma_2(TM;\mu)$ has a natural structure of $L^2(\mu)$-normed
$L^\infty(\mu)$-module if endowed with the usual vector space structure and
the following pointwise operations:
\begin{equation}\label{eq:concrete_modules_operations}\begin{split}
(fv)(x)&:=f(x)v(x)\in T_x M\\
|v|(x)&:=F\big(x,v(x)\big)
\end{split}
\quad\text{ for }\mu\text{-a.e.\ }x\in M,
\end{equation}
for any $v\in\Gamma_2(TM;\mu)$ and $f\in L^\infty(\mu)$. Similarly,
$\Gamma_2(T^*M;\mu)$ is an $L^2(\mu)$-normed $L^\infty(\mu)$-module.
\medskip

Standard verifications show that the modules $\Gamma_2(TM;\mu)$ and
$\Gamma_2(T^*M;\mu)$ have local dimension equal to $n:=\dim(M)$,
whence they are separable by \cite[Remark 5]{LP18}.
Furthermore, it holds that
\begin{equation}\label{eq:concrete_modules_dual}
\Gamma_2(T^*M;\mu)\text{ and }\Gamma_2(TM;\mu)
\text{ are one the module dual of the other,}
\end{equation}
in particular they are both reflexive as Banach spaces by \cite[Corollary 1.2.18]{G18_NonSmooth}.
It can also be readily proved that
\begin{equation}\label{eq:udf_generate}
\big\{\ud f\,:\,f\in C^1_c(M)\big\}\text{ generates }\Gamma_2(T^*M;\mu)
\text{ in the sense of modules,}
\end{equation}
where each element $\ud f$ can be viewed as an element of $\Gamma_2(T^*M;\mu)$
as it is a continuous section of the cotangent bundle $T^*M$ and its associated
pointwise norm $|\ud f|$ has compact support.
\begin{remark}{\rm
We emphasise that in \eqref{eq:udf_generate} it is fundamental to consider
$C^1$-functions (as opposed to Lipschitz functions).
The reason is that $C^1$-functions are everywhere differentiable,
thus in particular $\mu$-almost everywhere differentiable
(independently of the chosen measure $\mu$), while a Lipschitz function
might be not differentiable at any point of a set of positive $\mu$-measure.
\fr}\end{remark}
\begin{lemma}\label{lem:Gamma_for_Riemannian}
If $(M,F)$ is a Riemannian manifold, then $\Gamma_2(TM;\mu)$ is a Hilbert module.
Conversely, if $\Gamma_2(TM;\mu)$ is a Hilbert module and ${\rm spt}(\mu)=M$,
then $(M,F)$ is a Riemannian manifold.
\end{lemma}
\begin{proof} First of all, suppose that $(M,F)$ is a Riemannian manifold,
i.e.\ that each norm $F(x,\cdot)$ satisfies the parallelogram identity.
Then for any $v,w\in\Gamma_2(TM;\mu)$ it holds that
\[\begin{split}
|v+w|^2(x)+|v-w|^2(x)&=F\big(x,(v+w)(x)\big)^2+F\big(x,(v-w)(x)\big)^2\\
&=2\,F\big(x,v(x)\big)^2+2\,F\big(x,w(x)\big)^2\\
&=2\,|v|^2(x)+2\,|w|^2(x)\quad\text{ for }\mu\text{-a.e.\ }x\in M,
\end{split}\]
thus showing that $\Gamma_2(TM;\mu)$ is a Hilbert module.

Now suppose that the concrete tangent module $\Gamma_2(TM;\mu)$ is a Hilbert
module and ${\rm spt}(\mu)=M$. Let $U$ be the domain of some chart on $M$.
Then one can easily build a sequence $(v_i)_{i\in\N}$ of continuous vector fields
on $U$ such that
\begin{equation}\label{eq:v_i(x)_dense}
\big(v_i(x)\big)_{i\in\N}\text{ is dense in }T_x M\quad\text{ for every }x\in U.
\end{equation}
Hence for $\mu$-a.e.\ $x\in U$ we have that the identity
\[\begin{split}
F\big(x,(v_i+v_j)(x)\big)^2+F\big(x,(v_i-v_j)(x)\big)^2
&=|v_i+v_j|^2(x)+|v_i-v_j|^2(x)\\
&=2\,|v_i|^2(x)+2\,|v_j|^2(x)\\
&=2\,F\big(x,v_i(x)\big)^2+2\,F\big(x,v_j(x)\big)^2
\end{split}\]
holds for every $i,j\in\N$. Since the function $F:\,TM\to[0,+\infty)$ is continuous
and any set of full $\mu$-measure is dense in $M$,
we deduce from property \eqref{eq:v_i(x)_dense} that the norm $F(x,\cdot)$ satisfies
the parallelogram identity for every point $x\in U$. By arbitrariness of $U$, we thus
conclude that $(M,F)$ is a Riemannian manifold.
\end{proof}
\subsection{The isometric embedding
\texorpdfstring{$L^2_\mu(TM)\hookrightarrow\Gamma_2(TM;\mu)$}{iota}}
The aim of this conclusive subsection is to investigate the relation
between the abstract (co)tangent module and the concrete one.
The argument goes as follows: the natural projection map 
$\P:\,\Gamma_2(T^*M;\mu)\to L^2_\mu(T^*M)$ (Lemma \ref{lem:def_P})
is a quotient map (Proposition \ref{prop:P_quotient}), whence its
adjoint operator $\iota:\,L^2_\mu(TM)\to\Gamma_2(TM;\mu)$ is an
isometric embedding (Theorem \ref{thm:iota_isometric}). As a consequence,
the Sobolev space $W^{1,2}(M,\sfd_F,\mu)$ is a Hilbert space as soon
as $(M,F)$ is a Riemannian manifold (Theorem \ref{thm:UiH_Riemannian}).
Such results are essentially taken from the paper \cite{GP16_Behaviour}, where the
Euclidean case has been treated; anyway, we provide here their full
proof for completeness.
\medskip

In light of inequality \eqref{eq:|df|_leq_|udf|}, there is a natural
projection operator $\P$ from the concrete cotangent module
$\Gamma_2(T^*M;\mu)$ to the abstract cotangent module $L^2_\mu(T^*M)$.
The characterisation of such operator is the subject of the following result.
\begin{lemma}[The projection $\P$]\label{lem:def_P}
There exists a unique $L^\infty(\mu)$-linear and continuous operator
\begin{equation}\label{eq:map_P}
\P:\,\Gamma_2(T^*M;\mu)\longrightarrow L^2_\mu(T^*M)
\end{equation}
such that $\P(\ud f)=\d f$ for every $f\in C^1_c(M)$. Moreover, it holds that
\begin{equation}\label{eq:ineq_P}
\big|\P(\underline\omega)\big|\leq|\underline\omega|\;\;\;\mu\text{-a.e.}
\quad\text{ for every }\underline\omega\in\Gamma_2(T^*M;\mu).
\end{equation}
\end{lemma}
\begin{proof}
We denote by $\mathcal V$ the vector space of all elements of
$\Gamma_2(T^*M;\mu)$ that can be written in the form
$\sum_{i=1}^k\nchi_{E_i}\,\ud f_i$, where
$(E_i)_{i=1}^k$ is a Borel partition of $M$ and $(f_i)_{i=1}^k\subseteq C^1_c(M)$.
Recall that $\mathcal V$ is dense in $\Gamma_2(T^*M;\mu)$ by \eqref{eq:udf_generate}.
Since $\P$ is required to be $L^\infty(\mu)$-linear and to satisfy
$\P(\ud f)=\d f$ for all $f\in C^1_c(M)$, we are forced to set
\begin{equation}\label{eq:def_P}
\P(\underline\omega):=\sum_{i=1}^k\nchi_{E_i}\,\d f_i
\quad\text{ for every }\underline\omega=\sum_{i=1}^k\nchi_{E_i}\,\ud f_i\in\mathcal V.
\end{equation}
The well-posedness of such definition stems from the
validity of the $\mu$-a.e.\ inequality
\begin{equation}\label{eq:P_well-posed}
\bigg|\sum_{i=1}^k\nchi_{E_i}\,\d f_i\bigg|
=\sum_{i=1}^k\nchi_{E_i}\,|\d f_i|\overset{\eqref{eq:|df|_leq_|udf|}}\leq
\sum_{i=1}^k\nchi_{E_i}\,|\ud f_i|=
\bigg|\sum_{i=1}^k\nchi_{E_i}\,\ud f_i\bigg|,
\end{equation}
which also ensures that the map $\P:\,\mathcal V\to L^2_\mu(T^*M)$ is linear
continuous and accordingly can be uniquely extended to a linear continuous
operator $\P:\,\Gamma_2(T^*M;\mu)\to L^2_\mu(T^*M)$. Another consequence of
\eqref{eq:P_well-posed} is that $\P$ satisfies the inequality \eqref{eq:ineq_P}.
Finally, by suitably approximating any element of the space $L^\infty(\mu)$ with
a sequence of simple functions, we deduce from \eqref{eq:def_P} that the map $\P$
is $L^\infty(\mu)$-linear. This completes the proof of the statement.
\end{proof}

Given any $\omega\in L^2_\mu(T^*M)$, we infer from \eqref{eq:ineq_P} that
$|\omega|\leq|\underline\omega|$ holds $\mu$-a.e.\ for any
$\underline\omega\in\Gamma_2(T^*M;\mu)$ such that $\P(\underline\omega)=\omega$,
so that the estimate
\begin{equation}\label{eq:ineq_P_bis}
|\omega|\leq\underset{\underline\omega\in\P^{-1}(\omega)}{\rm ess\,inf\,}
|\underline\omega|\quad\text{ holds }\mu\text{-a.e.\ in }M.
\end{equation}
The next result shows that the inequality in \eqref{eq:ineq_P_bis} is actually
an equality, thus proving that the operator $\P$ is a \emph{quotient map}.
The proof relies upon Proposition \ref{prop:density_energy_C1} above.
\begin{proposition}[$\P$ is a quotient map]\label{prop:P_quotient}
The operator $\P$ satisfies the following property:
\begin{equation}\label{eq:P_quotient_claim}
\text{For any }\omega\in L^2_\mu(T^*M)\text{ there exists }
\underline\omega\in\P^{-1}(\omega)\text{ such that }
|\omega|=|\underline\omega|\text{ in the }\mu\text{-a.e.\ sense.}
\end{equation}
In particular, it holds that the map $\P$ is surjective.
\end{proposition}
\begin{proof}
We divide the proof into three steps:
\smallskip

\noindent{\color{blue}\textsc{Step 1: \eqref{eq:P_quotient_claim} for $\omega=\d f$.}}
Let $f\in W^{1,2}(M,\sfd_F,\mu)$ be fixed. By Proposition \ref{prop:density_energy_C1},
we can pick a sequence $(f_k)_{k\in\N}\subseteq C^1_c(M)$ such that $f_k\to f$
and $|\ud f_k|\to|\d f|$ in $L^2(\mu)$. In particular, it holds that $(\ud f_k)_{k\in\N}$
is bounded in $\Gamma_2(T^*M;\mu)$. Since $\Gamma_2(T^*M;\mu)$ is reflexive,
we have (up to subsequence) that $\ud f_k\weakto\underline\omega$ weakly in
$\Gamma_2(T^*M;\mu)$ for some $\underline\omega\in\Gamma_2(T^*M;\mu)$.
The map $\P$ being linear and continuous, it holds
$\d f_k=\P(\ud f_k)\weakto\P(\underline\omega)$ weakly in $L^2_\mu(T^*M)$.
Then Proposition \ref{prop:closure_diff} grants that $\P(\underline\omega)=\d f$.
Moreover, the $\mu$-a.e.\ inequality
$|\d f|=\big|\P(\underline\omega)\big|\leq|\underline\omega|$ follows from
\eqref{eq:ineq_P}. Hence from
${\big\||\underline\omega|\big\|}_{L^2(\mu)}\leq
\limi_k{\big\||\ud f_k|\big\|}_{L^2(\mu)}
={\big\||\d f|\big\|}_{L^2(\mu)}$
we deduce that $|\d f|=|\underline\omega|$ is satisfied in the $\mu$-a.e.\ sense.
This proves the claim \eqref{eq:P_quotient_claim} for all $\omega=\d f$ with
$f\in W^{1,2}(M,\sfd_F,\mu)$.
\smallskip

\noindent{\color{blue}\textsc{Step 2: \eqref{eq:P_quotient_claim} for $\omega$ simple.}}
Let $\omega\in L^2_\mu(T^*M)$ be of the form
$\omega=\sum_{i=1}^k\nchi_{E_i}\,\d f_i$, where $(E_i)_{i=1}^k$ is a Borel
partition of $M$ and $(f_i)_{i=1}^k\subseteq W^{1,2}(M,\sfd_F,\mu)$.
From \textsc{Step 1} we know that there exist elements
$\underline\omega_1,\ldots,\underline\omega_k\in\Gamma_2(T^*M;\mu)$ such
that $\P(\underline\omega_i)=\d f_i$ and $|\d f_i|=|\underline\omega_i|$
$\mu$-a.e.\ for all $i=1,\ldots,k$. Now call
$\underline\omega:=\sum_{i=1}^k\nchi_{E_i}\,\underline\omega_i\in\Gamma_2(T^*M;\mu)$.
Then the $L^\infty(\mu)$-linearity of $\P$ ensures that $\P(\underline\omega)=\omega$,
which together with the $\mu$-a.e.\ identity
\[
|\omega|=\bigg|\sum_{i=1}^k\nchi_{E_i}\,\d f_i\bigg|=
\sum_{i=1}^k\nchi_{E_i}\,|\d f_i|=\sum_{i=1}^k\nchi_{E_i}\,|\underline\omega_i|=
\bigg|\sum_{i=1}^k\nchi_{E_i}\,\underline\omega_i\bigg|=|\underline\omega|
\]
grant that the claim \eqref{eq:P_quotient_claim} holds whenever
$\omega$ is a simple $1$-form.
\smallskip

\noindent{\color{blue}\textsc{Step 3: \eqref{eq:P_quotient_claim} for general $\omega$.}}
Fix $\omega\in L^2_\mu(T^*M)$. Since simple $1$-forms are dense in $L^2_\mu(T^*M)$,
we can choose a sequence $(\omega_k)_{k\in\N}\subseteq L^2_\mu(T^*M)$ of simple
$1$-forms converging to $\omega$. Given $k\in\N$, there exists an element
$\underline\omega_k\in\Gamma_2(T^*M;\mu)$ such that $\P(\underline\omega_k)=\omega_k$
and $|\omega_k|=|\underline\omega_k|$ $\mu$-a.e.\ by \textsc{Step 2}.
In particular, the sequence $(\underline\omega_k)_{k\in\N}$ is bounded in the
reflexive space $\Gamma_2(T^*M;\mu)$, whence there exists
$\underline\omega\in\Gamma_2(T^*M;\mu)$ such that (up to subsequence)
we have $\underline\omega_k\weakto\underline\omega$. Since $\P$ is linear and
continuous, we deduce that $\omega_k=\P(\underline\omega_k)\weakto\P(\underline\omega)$.
On the other hand, it holds that $\omega_k\to\omega$ by assumption,
thus necessarily $\P(\underline\omega)=\omega$. Finally, we have
$|\omega|=\big|\P(\underline\omega)\big|\leq|\underline\omega|$ $\mu$-a.e.\ by
\eqref{eq:ineq_P} and
\[
{\big\||\underline\omega|\big\|}_{L^2(\mu)}\leq
\limi_{k\to\infty}{\big\||\underline\omega_k|\big\|}_{L^2(\mu)}=
\limi_{k\to\infty}{\big\||\omega_k|\big\|}_{L^2(\mu)}=
{\big\||\omega|\big\|}_{L^2(\mu)},
\]
so that $|\omega|=|\underline\omega|$ in the $\mu$-a.e.\ sense.
This shows the claim \eqref{eq:P_quotient_claim} for any $\omega\in L^2_\mu(T^*M)$.
\end{proof}

Our main result is the following: the adjoint operator $\iota$ of the map $\P$
is an isometric embedding of the abstract tangent module $L^2_\mu(TM)$ into
the concrete tangent module $\Gamma_2(TM;\mu)$. This is achieved by duality in
the ensuing theorem, as a consequence of the fact that $\P$ is a quotient map.
\begin{theorem}[The isometric embedding $\iota$]\label{thm:iota_isometric}
Let $(M,F)$ be a geodesically complete, reversible Finsler manifold and
$\mu$ a non-negative Radon measure on $(M,\sfd_F)$. Let us denote by
\begin{equation}
\iota:\,L^2_\mu(TM)\longrightarrow\Gamma_2(TM;\mu)
\end{equation}
the adjoint map of $\P:\,\Gamma_2(T^*M;\mu)\to L^2_\mu(T^*M)$,
i.e.\ the unique $L^\infty(\mu)$-linear and continuous operator satisfying
\begin{equation}\label{eq:def_iota}
\underline\omega\big(\iota(v)\big)=\P(\underline\omega)(v)\;\;\;\mu\text{-a.e.}
\quad\text{ for every }v\in L^2_\mu(TM)\text{ and }\underline\omega\in\Gamma_2(T^*M;\mu).
\end{equation}
Then it holds that
\begin{equation}\label{eq:iota_isometric}
\big|\iota(v)\big|=|v|\;\;\;\mu\text{-a.e.}\quad\text{ for every }v\in L^2_\mu(TM).
\end{equation}
In particular, the operator $\iota$ is an isometric embedding
and $L^2_\mu(TM)$ is a finitely-generated module.
\end{theorem}
\begin{proof}
First of all, the $\mu$-a.e.\ inequality
$\big|\P(\underline\omega)(v)\big|\leq|\underline\omega||v|$ -- which is granted
by \eqref{eq:ineq_P} -- shows that the element $\iota(v)$ in \eqref{eq:def_iota}
is well-defined and that the map $\iota$ is $L^\infty(\mu)$-linear continuous.
The same inequality also implies that $\big|\iota(v)\big|\leq|v|$ holds
$\mu$-a.e.\ for any fixed $v\in L^2_\mu(TM)$. On the other hand, pick any
$\omega\in L^2_\mu(T^*M)$ such that $|\omega|\leq 1$ $\mu$-a.e.\ in $M$.
Proposition \ref{prop:P_quotient} provides us with some element
$\underline\omega\in\Gamma_2(T^*M;\mu)$ satisfying $\P(\underline\omega)=\omega$
and $|\underline\omega|=|\omega|$ in the $\mu$-a.e.\ sense. Therefore
\[
\omega(v)=\P(\underline\omega)(v)\leq
\underset{|\underline\omega'|\leq 1}{\rm ess\,sup\;}\P(\underline\omega')(v)
\overset{\eqref{eq:def_iota}}=
\underset{|\underline\omega'|\leq 1}{\rm ess\,sup\;}\underline\omega'\big(\iota(v)\big)=
\big|\iota(v)\big|\quad\mu\text{-a.e.\ in }M,
\]
whence accordingly we conclude that
\[
|v|=\underset{|\omega|\leq 1}{\rm ess\,\sup\;}\omega(v)\leq\big|\iota(v)\big|
\quad\mu\text{-a.e.\ in }M.
\]
This proves that the identity \eqref{eq:iota_isometric} is satisfied.
The last statement now directly follows from the fact that the module
$\Gamma_2(TM;\mu)$ has local dimension equal to $n$.
\end{proof}
\begin{remark}{\rm
In general, the isometric embedding $\iota:\,L^2_\mu(TM)\to\Gamma_2(TM;\mu)$
provided by Theorem \ref{thm:iota_isometric} might not be an isomorphism
(even if we assume that the measure $\mu$ has full support).

For instance, choose any sequence $(a_k)_{k\in\N}$ of positive real numbers
such that $\sum_{k=0}^\infty a_k<+\infty$, enumerate the rational numbers as
$(q_k)_{k\in\N}$ and define the finite Borel measure $\mu$ on $\R$ as
\[\mu:=\sum_{k=0}^\infty a_k\,\delta_{q_k},\quad\text{ where }\delta_{q_k}
\text{ is the Dirac delta at }q_k.\]
Therefore $W^{1,2}(\R,\sfd_{\rm Eucl},\mu)=L^2(\mu)$ and all its elements have
null minimal weak upper gradient (cf.\ \cite[Remark 4.12]{AGS11_HeatFlow}),
thus accordingly $L^2_\mu(T\R)=\{0\}$. On the other hand, it is immediate
to verify that the space $\Gamma_2(T\R;\mu)$ is non-trivial.
\fr}\end{remark}
\begin{corollary}\label{cor:Sobolev_reflexive}
Let $(M,F)$ be a geodesically complete, reversible Finsler manifold.
Let $\mu$ be a non-negative Radon measure on $(M,\sfd_F)$. Then the
Sobolev space $W^{1,2}(M,\sfd_F,\mu)$ is reflexive.
\end{corollary}
\begin{proof}
Theorem \ref{thm:iota_isometric} says that $L^2_\mu(TM)$ is
finitely-generated, whence it is reflexive
(cf.\ for instance \cite[Theorem 1.4.7]{G18_NonSmooth}).
This implies that $W^{1,2}(M,\sfd_F,\mu)$ is reflexive by
Proposition \ref{prop:weak_cptness_d}.
\end{proof}
\begin{remark}{\rm
We point out that Corollary \ref{cor:Sobolev_reflexive} can be alternatively
deduced as a consequence of a result by Ambrosio-Colombo-Di Marino,
namely \cite[Corollary 7.5]{ACDM}, as we are going to sketch.

Fix $\bar x\in M$. We call $M_k$ the closed ball of radius $k\in\N$ centered
at $\bar x$ and we set $\mu_k:=\mu\restr{M_k}$. By properness of $(M,\sfd_F)$
and Bishop-Gromov inequality we know that each metric space $(M_k,\sfd_F)$
is doubling, thus accordingly $W^{1,2}(M_k,\sfd_F,\mu_k)$ is reflexive for
every $k\in\N$ by \cite[Corollary 7.5]{ACDM}.
Now pick a bounded sequence $(f_i)_{i\in\N}$ in $W^{1,2}(M,\sfd_F,\mu)$.
A diagonalisation argument, together with Proposition \ref{prop:weak_cptness_d}
and  \cite[Proposition 2.6]{Gigli12}, grant the existence of a function
$f\in W^{1,2}(M,\sfd_F,\mu)$ and of a (not relabeled) subsequence of $(f_i)_{i\in\N}$
such that $(f_i,\d f_i)\weakto(f,\d f)$ in the weak topology of
$L^2(\mu)\times L^2_\mu(T^*M)$. This yields the reflexivity of $W^{1,2}(M,\sfd_F,\mu)$
by Proposition \ref{prop:weak_cptness_d}.
\fr}\end{remark}

We conclude by focusing on the special case of Riemannian manifolds.
By combining Theorem \ref{thm:iota_isometric} with Lemma
\ref{lem:Gamma_for_Riemannian}, we can immediately obtain the following result.
(A word on notation: given a Riemannian manifold $(M,g)$, we
denote by $\sfd_g$ the distance induced by the metric $g$.)
\begin{theorem}[Weighted Riemannian manifolds are infinitesimally Hilbertian]
\label{thm:UiH_Riemannian}
Let $(M,g)$ be a geodesically complete Riemannian manifold.
Fix any non-negative Radon measure $\mu$ on $(M,\sfd_g)$.
Then the metric measure space $(M,\sfd_g,\mu)$ is infinitesimally Hilbertian.
\end{theorem}
\begin{proof}
Let us define $F(x,v):=g_x(v,v)^{1/2}$ for every $x\in M$ and $v\in T_x M$,
so that $(M,F)$ is a reversible Finsler manifold (and $\sfd_F=\sfd_g$).
Consider the embedding $\iota:\,L^2_\mu(TM)\to\Gamma_2(TM;\mu)$
provided by Theorem \ref{thm:iota_isometric}. Since $\iota$ preserves the
pointwise norm and $\Gamma_2(TM;\mu)$ is a Hilbert module (by Lemma
\ref{lem:Gamma_for_Riemannian}), we deduce that $L^2_\mu(TM)$ is a Hilbert
module as well. This grants that the Sobolev space $W^{1,2}(M,\sfd_F,\mu)$
is a Hilbert space by \eqref{eq:characterisation_W12_Sobolev}, thus proving
the statement.
\end{proof}
\section{Alternative proof of Theorem \ref{thm:UiH_Riemannian}}\label{s:alternative_proof}
Here we provide an alternative proof of Theorem \ref{thm:UiH_Riemannian}.
Instead of deducing it as a corollary of Theorem \ref{thm:iota_isometric},
we rather make use of the following fact that has been achieved in \cite{GP16_Behaviour}:
\begin{equation}\label{eq:main_thm_linear_case}\begin{split}
&\text{Let }\big(V,\langle\cdot,\cdot\rangle\big)
\text{ be a finite-dimensional scalar product space. Let }\nu\geq 0\text{ be any}\\
&\text{Radon measure on }\big(V,\langle\cdot,\cdot\rangle\big).\text{ Then }
\big(V,\langle\cdot,\cdot\rangle,\nu\big)\text{ is infinitesimally Hilbertian.}
\end{split}\end{equation}
(Actually, the result is proven for $V=\R^d$ equipped with the Euclidean distance,
but -- as observed in \cite[Remark 2.11]{GP16_Behaviour} -- the very same proof
works for any finite-dimensional scalar product space.)
\medskip

Let $f,g\in W^{1,2}(M,\sfd_g,\mu)$ be fixed. 
In order to prove the claim, it is enough to show that
\begin{equation}\label{eq:parallelogram}
\big|D(f+g)\big|^2+\big|D(f-g)\big|^2=2\,|Df|^2+2\,|Dg|^2
\quad\text{ holds }\mu\text{-a.e.\ on }M. 
\end{equation}
Fix any $\varepsilon>0$. By using Theorem \ref{thm:Deng-Hou}
and the Lindel\"of property of $(M,\sfd_g)$, we can find two sequences
$(x_i)_{i\in\N}\subseteq M$ and $(r_i)_{i\in\N}\subseteq(0,+\infty)$
satisfying the following properties:
\begin{itemize}
\item[$\rm i)$] Calling $V_i$ the closed ball in $(M,\sfd_g)$ having radius
$r_i$ and center $x_i$, it holds that $(V_i)_{i\in\N}$ is a cover of $M$.
\item[$\rm ii)$] Calling $W_i$ the closed ball in $(T_{x_i}M,g_{x_i})$
having radius $r_i$ and center $0$, it holds that each exponential map 
$\exp_{x_i}$ is $(1+\varepsilon)$-biLipschitz between $W_i$ and $V_i$.
\end{itemize}
For any $i\in \N$, let us denote by $\varphi_i:\,V_i\to W_i$ the inverse
map of $\exp_{x_i}\restr{W_i}$. Define $\mu_i:=\mu\restr{V_i}$ and
$\nu_i:=(\varphi_i)_*\mu_i$. Then $\varphi_i$ is a map of bounded deformation 
from $(V_i,\sfd_g,\mu_i)$ to $(W_i,g_{x_i},\nu_i)$, with inverse of bounded
deformation (cf.\ \cite[Definition 2.4.1]{G18_NonSmooth} for the notion of map
of bounded deformation). Therefore \cite[formula (2.4.1)]{G18_NonSmooth} ensures
that for every $h\in W^{1,2}(V_i,\sfd_g,\mu_i)$ one has that
$h\circ\varphi_i^{-1}\in W^{1,2}(W_i,g_{x_i},\nu_i)$ and that
\begin{equation}\label{eq:(1+epsilon)_bound_on_mwug}
\frac{|Dh|\circ\varphi_i^{-1}}{1+\varepsilon}
\leq\big|D(h\circ\varphi_i^{-1})\big|\leq
(1+\varepsilon)\,|Dh|\circ\varphi_i^{-1}
\quad\text{ holds }\nu_i\text{-a.e.\ on }T_{x_i}M.
\end{equation}
Furthermore, we know from \cite[Proposition 2.6]{Gigli12} that
for any $h\in W^{1,2}(M,\sfd_g,\mu)$ and $i\in\N$ one has that
$\nchi_{V_i}h\in W^{1,2}(V_i,\sfd_g,\mu_i)$ and that
\begin{equation}\label{eq:mwug_localised}
\big|D({\nchi_{V_i}h})\big|=|Dh|\quad\text{ holds }\mu_i\text{-a.e.\ on }V_i.
\end{equation}
Now let us set $f_i:=(\nchi_{V_i}f)\circ \varphi_i^{-1}$ and
$g_i:=(\nchi_{V_i}g)\circ \varphi_i^{-1}$ for every $i\in \N$. We have that
the Sobolev space $W^{1,2}(T_{x_i}M,g_{x_i},\nu_i)\simeq W^{1,2}(W_i,g_{x_i},\nu_i)$
is a Hilbert space by \eqref{eq:main_thm_linear_case}, whence accordingly
\begin{equation}\label{eq:parallelogram_i}
\big|D(f_i+g_i)\big|^2+\big|D(f_i-g_i)\big|^2=2\,|Df_i|^2+2\,|Dg_i|^2
\quad\text{ holds }\nu_i\text{-a.e.\ on }T_{x_i}M.
\end{equation}
By combining \eqref{eq:(1+epsilon)_bound_on_mwug}, \eqref{eq:mwug_localised}
and \eqref{eq:parallelogram_i} we conclude that
\begin{equation}\label{eq:(1+epsilon)_parallelogram}
\frac{2\,|Df|^2+2\,|Dg|^2}{(1+\varepsilon)^4}\leq
\big|D(f+g)\big|^2+\big|D(f-g)\big|^2\leq
(1+\varepsilon)^4\,\big(2\,|Df|^2+2\,|Dg|^2\big)
\end{equation}
holds $\mu_i$-a.e.\ for every $i\in \N$. This implies that
\eqref{eq:(1+epsilon)_parallelogram} is satisfied
$\mu$-a.e.\ on $M$, so by letting $\varepsilon\searrow 0$
we finally obtain \eqref{eq:parallelogram}, as required.
\begin{remark}{\rm
It seems to us that also Theorem \ref{thm:iota_isometric} could
be deduced from the results in \cite{GP16_Behaviour} via a suitable
localisation argument, but with a much more involved proof. For this
reason, we chose the presentation seen in Section \ref{s:main} above.
\fr}\end{remark}
\end{document}